\documentclass[12pt]{article}
\usepackage{amsmath,fullpage,amssymb,amsthm,hyperref,enumerate}
\usepackage{tikz,color}
 \usetikzlibrary{patterns}
 \usepackage{diagbox}
 
\newtheorem{theorem}{Theorem}[section]
\newtheorem{proposition}[theorem]{Proposition}
\newtheorem{lemma}[theorem]{Lemma}
\newtheorem{corollary}[theorem]{Corollary}

\theoremstyle{definition}
\newtheorem{example}[theorem]{Example}

\theoremstyle{remark}
\newtheorem{remark}[theorem]{Remark}  

\renewcommand\S{\mathcal S}
\DeclareMathOperator\lrm{lrm}
\newcommand\e{\mathrm e}
\renewcommand\Pr{\mathbb{P}}
\newcommand\E{\mathbb{E}}
\newcommand\W{\mathcal W}
\newcommand\U{\mathcal U}
\newcommand\Uc{\overline\U}
\newcommand\V{\mathcal V}
\newcommand\Vc{\overline\V}
\newcommand\uc{\overline u}
\newcommand\vc{\overline v}
\newcommand\tR{\widetilde R}
\newcommand\delete[1]{}
\newcommand\Q{\mathbb{Q}}

\newcommand\sinsq{\sin\!\left(\frac{\sqrt{3}x}{2}\right)}
\newcommand\cossq{\cos\!\left(\frac{\sqrt{3}x}{2}\right)}

\title{On a dart game of Niedermaier}

\author{Sergi Elizalde}

\date{}

\begin{document}

\maketitle

\begin{abstract}
We analyze a game introduced by Andy Niedermaier, where $p$ players take turns throwing a dart at a dartboard. A player is eliminated unless his dart lands closer to the center than all previously thrown darts, in which case he goes to the back of the line, until only one player remains. 
Using generating functions, we determine the distribution of the number of throws in the game, and we obtain a recursive formula to compute the probability that each player wins. 

\end{abstract}

\section{Introduction}

Consider the game where $p$ players are lined up, from player $1$ in the front to player $p$ in the back, to take turns throwing a dart at a dartboard.
When a player throws a dart, two things can happen:
\begin{itemize}
\item if the dart lands closer to the center than all previously thrown darts, the player stays in the game and moves to the back of the line;
\item otherwise, the player is eliminated.
\end{itemize}
The game ends when all the players but one have been eliminated, in which case the remaining player is the winner.

\begin{example}\label{ex1}
A 3-player game could go as follows. Player 1's dart lands at distance 6, and he moves to the back of the line. Then player 2's dart lands at distance 8, so player 2 is eliminated. Then player 3's dart lands at distance 5, so she moves to the back of the line.
It's now player 1's turn, and his dart lands at distance 4. Then player 3's dart lands at distance 2, and finally player 1's dart lands at distance 7, so he is eliminated, making player 3 the winner. 
\end{example}

The players are assumed to be equally skilled, so the distances between each of the thrown darts and the center of the dartboard can be modeled by a sequence or i.i.d.\ random variables. We assume that the probability that two darts land at the same distance of the center is zero.

We heard of this game from Andy Niedermaier~\cite{Nie}, who in turn learned about it from his graduate advisor Jeff Remmel. In unpublished work,
Niedermaier obtained formulas for the winning probability of each player when the game has at most $4$ players.

In this paper we extend his work and analyze additional aspects of the game. After relating the game to permutation statistics in Section~\ref{sec:Stirling}, in Section~\ref{sec:length}
we determine the average number of throws that take place in a $p$-player game, and more generally the distribution of this number of throws, expressed as a compact generating function. We then extend our method to study the number of further throws left at a given time in the game, knowing how many players remain and what is the best throw so far.
In Section~\ref{sec:winner}, addressing a question of Niedermaier~\cite{Nie}, we describe a recursive formula for the generating function giving the probability that each player wins the $p$-player game, refined by the number of throws (which we call the {\em length} of the game). 
Finally, in Section~\ref{sec:monotone} we show that, for a fixed number of players $p$, the probability that player $k$ wins is a decreasing function of $k$. 

Table~\ref{tab:notation} summarizes some of the notation used throughout the paper.

\begin{table}[h] \centering
\begin{tabular}{c|c|c}
parameter & description & variable \\ \hline
$n$ & number of dart throws & $z$\\
$p$ & number of players & $u$ \\
$k$ & winning player & $t$ \\
$x$ & distance of best throw so far & $x$
\end{tabular} \bigskip

\begin{tabular}{c|l}
notation & probability that $p$-player game ...  \\ \hline
$R_{n,p,k}$ & ... is won by player $k$ after $n$ throws \\
$Q_{n,p}=\sum_{k=1}^n R_{n,p,k}$ & ... ends after $n$ throws  \\
$P_{p,k}=\sum_{n\ge0} R_{n,p,k}$ & ... is won by player $k$  
\end{tabular} 
\caption{Summary of some parameters, variables, and other notation used in the paper.}
\label{tab:notation}
\end{table}

\section{Permutations and Stirling numbers of the first kind}\label{sec:Stirling}

Let $\S_n$ denote the set of permutations of $[n]=\{1,2,\dots,n\}$.
After $n$ darts have been thrown, their distances to the center are given by i.i.d.\ random variables $X_1,X_2,\dots,X_n$.
The relative order of these values determines a permutation $\pi\in\S_n$, whose one-line notation is obtained by replacing the smallest distance by a $1$, the second smallest distance by a $2$, and so on. This process generates each of the $n!$ permutations with the same probability.

We say that $i$ is a {\em left-to-right minimum} of $\pi$ if $\pi_i<\pi_j$ for all $j<i$. Let $\lrm(\pi)$ denote the number of left-to-right minima of $\pi$.
If $i\in[n]$ is not a left-to-right minimum of $\pi$, we say that $i$ is a {\em botch} of $\pi$. 
The player who makes the $i$th throw stays in the game if and only if $i$ is a left-to-right minimum in the corresponding permutation $\pi$, and is eliminated if $i$ is botch.
In Example~\ref{ex1}, the sequence of distances is $6,8,5,4,2,7$, and the corresponding permutation is $\pi=463215\in\S_6$, which has left-to-right minima in positions $1,3,4,5$.

It is well known \cite[Sec.\ 1.3]{EC1} that the number of permutations in $\S_n$ with $k$ left-to-right minima is equal to the signless Stirling number of the first kind $c(n,k)$, which also counts permutations in $\S_n$ with $k$ cycles. The exponential generating function for these numbers is
\begin{equation}\label{eq:GFStirling} 
\sum_{n,k\ge0} c(n,k)\, v^k \frac{z^n}{n!}=e^{v\log\frac{1}{1-z}}=\frac{1}{(1-z)^{v}}.
\end{equation}
Note that $c(n,k)=0$ if $k<0$ or $k>n$.

A complete game with $p\ge2$ players is encoded by a permutation that ends with a botch and has $p-1$ botches in total. Equivalently, by a permutation $\pi\in\S_n$ (for some $n\ge p$) having $n-p+1$ left-to-right minima, where position $n$ is not one of them (i.e., $\pi_n\neq 1$). 
Such a permutation is obtained uniquely by taking a permutation in $\S_{n-1}$ with $n-p+1$ left-to-right minima, appending an element $m\in\{2,3,\dots,n\}$, and shifting the other values accordingly (i.e., adding one to the entries $\ge m$).
Thus, the number of such permutations is
\begin{equation}\label{eq:complete} (n-1)\,c(n-1,n-p+1). \end{equation}

\section{The length of a game}\label{sec:length}

\subsection{The distribution of the number of throws}

For $p\ge1$, let $L_p$ be the random variable for the number of throws in a $p$-player game, and let 
$Q_{n,p}=\Pr(L_p=n)$, that is, the probability that a $p$-player game ends after exactly $n$ throws. Our next result gives a closed form for the generating function 
$$Q(u,z)=\sum_{\substack{n\ge0\\ p\ge1}} Q_{n,p} u^{p} z^n.$$

\begin{theorem}\label{thm:Quz} We have
$$Q(u,z)=u+\frac{u^2}{1-u}\left(1+\frac{z-1}{(1-uz)^{1/u}}\right).$$
\end{theorem}

\begin{proof}
For $p=1$, the game ends with zero throws, corresponding to the summand $u$.

For $p\ge2$, the game requires at least $2$ throws, since no player is eliminated after the first throw. Fix $n\ge2$, and let $X_1,X_2,\dots,X_n$ be the i.i.d.\ random variables that describe the distance of the first $n$ darts from the center, which determine a uniformly random permutation $\pi\in\S_n$. Then $\pi$ encodes a complete game with $p$ players if and only if $\lrm(\pi)=n-p+1$ and $\pi(n)\neq 1$. The number of such permutations is given by equation~\eqref{eq:complete}. Any $\pi\in\S_n$ that is not of this form would encode a game where either $p-1$ players get eliminated before the $n$th throw takes place (and thus the game ends before the $n$th throw, making the remaining throws irrelevant), or fewer than $p-1$ players have been eliminated when the $n$th throw takes place (and thus the game has not ended at this point yet).

It follows that the probability that the game ends after exactly $n$ throws is 
$$Q_{n,p}=\frac{(n-1)\,c(n-1,n+1-p)}{n!}.$$ 
Summing over $n$ and $p$,
\begin{align*}Q(u,z)&=u+ \sum_{n,p\ge2} (n-1)\,c(n-1,n+1-p)\, u^{p}\frac{z^n}{n!}\\
&=u+\sum_{\substack{n\ge1\\ p\ge2}} c(n-1,n+1-p)\, u^{p}\frac{z^n}{(n-1)!}-\sum_{\substack{n\ge1\\ p\ge2}} c(n-1,n+1-p)\, u^{p}\frac{z^n}{n!}.
\end{align*}
By equation~\eqref{eq:GFStirling},
$$
\sum_{\substack{n\ge1\\ p\ge2}} c(n-1,n+1-p)\, u^{p} \frac{z^{n-1}}{(n-1)!}=u^2 \sum_{n,k\ge0} c(n,k)\, u^{n-k} \frac{z^n}{n!}=\frac{u^2}{(1-uz)^{1/u}}.$$
Thus, we can write
\begin{align*}Q(u,z)
=u+\frac{u^2z}{(1-uz)^{1/u}}-\int_0^z \frac{u^2}{(1-ut)^{1/u}}\,dt
=u+\frac{u^2}{1-u}\left(1+\frac{z-1}{(1-uz)^{1/u}}\right).
\end{align*}
\end{proof}

Let now $E_{p}=\E(L_p)$, that is, the expected number of dart throws in a $p$-player game, and denote the corresponding generating function by
$$E(u)=\sum_{p\ge1} E_{p} u^{p}.$$ 

\begin{corollary} We have
$$E(u)=u^2(1-u)^{-1-1/u},$$
and so the expected length of a $p$-player game is asymptotically
$$E_p=p+\log p+ \gamma -1+o(1),$$
where $\gamma\approx0.5772$ is Euler's constant.
\end{corollary}

\begin{proof}
The generating function for $E_{p}=\sum_{n\ge0}n\,Q_{n,p}$ is given by
$$E(u)=\left.\frac{\partial Q(u,z)}{\partial z}\right|_{z=1}=u^2(1-u)^{-1-1/u},$$
by Theorem~\ref{thm:Quz}.

To determine the asymptotic behavior of the coefficients, we observe that $E(u)$ is analytic in the complex plane slit along $\mathbb{R}_{\ge1}$, and
compute its singular expansion at $u=1$:
$$E(u)=\frac{1}{(1-u)^2}+\frac{1}{1-u}\log \frac{1}{1-u}-\frac{2}{1-u}+1+o\left(\frac{1}{1-u}\right).$$
Then we use the transfer technique from Flajolet and Sedgewick \cite[Sec.\ VI.4]{FS} to asymptotically estimate the coefficient of $u^p$, similarly to how it is done in \cite[Ex.\ VI.4]{FS}.
\end{proof}

The first few terms of the expansion of this generating function are
$$E(u)=\e\,u^2+\frac{3}{2} \e\, u^3 +\frac{47}{24} \e\, u^{4}+\frac{115}{48} \e\, u^{5}+\frac{16247}{5760} \e\, u^{6}+\frac{37289}{11520} \e\, u^{7}+\frac{10587043}{2903040} \e\, u^{8}+\frac{2614099}{645120} \e\, u^{9}+\dots
$$

\subsection{The number of remaining throws}

Next we extend the above results to study the number of remaining throws at any given time during the game.
Suppose that, at a given point, there are $p$ players left, and $d$ is the distance to the center of the best throw so far, with the convention $d=\infty$ at the beginning of the game. Denote by $x=\Pr(X_i<d)\in[0,1]$ the probability that a given throw lands closer to the center, or equivalently the probability that the next player is not eliminated. Let $L_p(x)$ be the random variable for the number of throws remaining in this situation. By definition, $L_p(1)=L_p$, and it is easy to see that $L_p(0)=p-1$, since in the case that $x=0$, the next $p-1$ players to throw are eliminated and the game ends. 
In analogy to the definition of $Q_{n,p}$, let $Q_{n,p}(x)=\Pr(L_p(x)=n)$ be the probability that, when there are $p$ players left
and the probability that a given throw beats the best throw so far is $x$, the game ends after exactly $n$ further throws. For $p\ge1$, let
$$Q(x,u,z)=\sum_{\substack{n\ge0\\ p\ge1}}Q_{n,p}(x) u^pz^n.$$

\begin{remark}\label{rem:uniform}
Without loss of generality, we will assume from now on that the random variables $X_i$ that describe the distance of the throws to the center are uniformly distributed in $[0,1]$, in which case we can interpret $x$ as the smallest distance to the center of the best throw so far, that is, the distance to beat in order to stay in the game.
\end{remark}

\begin{lemma}\label{lem:Qnp}
For $n\ge1$ and $p\ge2$ , the polynomials $Q_{n,p}(x)$ satisfy the recurrence
\begin{equation}\label{eq:Qnp} Q_{n,p}(x)=(1-x)Q_{n-1,p-1}(x)+\int_0^x Q_{n-1,p}(v)\,dv,\end{equation}
with initial conditions $Q_{0,1}(x)=1$, $Q_{n,1}(x)=0$ for $n\ge1$, and $Q_{0,p}(x)=0$ for $p\ge2$.
\end{lemma}

\begin{proof}
To check the initial conditions, note that the game ends precisely when there is one player left. Thus, if $p=1$, the number of further throws is always $0$ (with probability $1$), and if $p\ge2$, the number of further throws is never $0$.

To prove equation~\eqref{eq:Qnp}, we consider two options for the player that throws next. If he gets eliminated, which happens with probability $1-x$, then the game reduces to a $(p-1)$-player game with one fewer throw left. Otherwise, his throw lands at distance $v$, for some $0\le v<x$, and the game reduces to a $p$-player game with one fewer throw left, where the distance to beat is $v$.

Finally, the recurrence and the initial conditions imply that $Q_{n,p}(x)$ is always a polynomial of degree at most~$n$.
\end{proof}

The polynomials $Q_{n,p}(x)$ for small values of $n$ and $p$ are given in Table~\ref{tab:Qnp}. From Lemma~\ref{lem:Qnp}, we obtain the following compact expression for the generating function $Q(x,u,z)$.

\begin{table}[htb]
$$\begin{array}{c|ccccc}
\hbox{\diagbox{$n$}{$p$}}& 1&2&3&4&5\\ \hline
0& 1&0&0&0&0\\
1& 0&1-x&0&0&0\\
2& 0&x-\dfrac{x^2}{2}&(1-x)^2&0&0\\
3& 0&\dfrac{x^2}{2}-\dfrac{x^3}{6}&2x-\dfrac{5x^2}{2}+\dfrac{5x^3}{6}&(1-x)^3&0\\
4& 0&\dfrac{x^3}{6}-\dfrac{x^4}{24}&\dfrac{3x^2(2-x)^2}{8}& 3x-6x^2+\dfrac{13x^3}{3}-\dfrac{13x^4}{12} &(1-x)^4
\end{array}$$
\caption{The polynomials $Q_{n,p}(x)$ giving the probability that, when $p$ players remain and the distance to beat is $x\in[0,1]$, the game ends after $n$ further throws, for $n\le4$ and $p\le5$.}
\label{tab:Qnp}
\end{table}

\begin{theorem}\label{thm:Qxuz} We have
$$Q(x,u,z)=\frac{u}{1-u}+\left(\frac{u}{1-uz}-\frac{u}{1-u}\right)\left(\frac{1-uz}{1-(1-x)uz}\right)^{1-{1/u}}.$$
\end{theorem}

\begin{proof}
Multiplying equation~\eqref{eq:Qnp} by $u^pz^n$ on both sides, summing over $n\ge1$ and $p\ge2$, and adding the initial conditions, we obtain
\begin{align*}Q(x,u,z) &= u+\sum_{\substack{n\ge1\\ p\ge2}} Q_{n,p}(x) u^pz^n \\
&= u+(1-x)\sum_{\substack{n\ge1\\ p\ge2}}  Q_{n-1,p-1}(x) u^pz^n +\int_0^x \sum_{\substack{n\ge1\\ p\ge2}} Q_{n-1,p}(v) u^pz^n \,dv\\
&= u+(1-x)uz Q(x,u,z) +\int_0^x z(Q(v,u,z)-u)  \,dv\\
&= u(1-xz)+(1-x)uz Q(x,u,z) +z \int_0^x Q(v,u,z) \,dv.
\end{align*}

Differentiating with respect to $x$, we get
$$(1-(1-x)uz)\frac{\partial Q(x,u,z)}{\partial x} =(1-u)zQ(x,u,z)-uz.$$
The initial condition $Q(0,u,z)=\frac{u}{1-uz}$ follows from Lemma~\ref{lem:Qnp} or directly from the fact that, with $x=0$ and $p$ players remaining, the game always ends after $p-1$ throws since every throw is a losing one. Solving this differential equation yields the stated expression for $Q(x,u,z)$.
\end{proof}

Setting $z=1$ in Theorem~\ref{thm:Qxuz} yields $Q(x,u,1)=\frac{u}{1-u}$, which is equivalent to $\sum_{n} Q_{n,p}(x)=1$ for all $p\ge1$, agreeing with the fact that $\{Q_{n,p}(x)\}_{n}$ is a probability distribution for fixed $p$ and $x\in[0,1]$.

Let now $E_{p}(x)=\E(L_p(x))$, that is, the expected number of remaining dart throws when there are $p$ players left and the distance to beat is $x\in[0,1]$. Let
$$E(x,u)=\sum_{p\ge1} E_{p}(x) u^{p}.$$ 

\begin{corollary} We have
$$E(x,u)=\left.\frac{\partial Q(x,u,z)}{\partial z}\right|_{z=1}=\frac{u^2}{(1-u)^2}\left(\frac{1-u}{1-(1-x)u}\right)^{1-{1/u}},$$
and so the expected number of remaining throws when there are $p$ players left and the distance to beat is $x\in[0,1]$ is asymptotically
$$E_p(x)=p+\log p+ \gamma -1+\log x + o(1).$$
\end{corollary}

\begin{proof}
By differentiating the expression in Theorem~\ref{thm:Qxuz}, we obtain the following.
$$E(x,u)=\left.\frac{\partial Q(x,u,z)}{\partial z}\right|_{z=1}=\frac{u^2}{(1-u)^2}\left(\frac{1-u}{1-(1-x)u}\right)^{1-{1/u}}.$$

The singular expansion at $u=1$ is now
$$E(x,u)=\frac{1}{(1-u)^2}+\frac{1}{1-u}\log \frac{1}{1-u}+\frac{\log x - 2}{1-u}+\frac{1}{x}+o\left(\frac{1}{1-u}\right),$$
which yields the stated asymptotic estimate for the coefficient of $u^p$.
\end{proof}

\section{Refining by the winning player}\label{sec:winner}

Next we refine the generating functions in Section~\ref{sec:length} to keep track of the winner of the game. For $1\le k\le p$, let $P_{p,k}$ denote the probability that player $k$ wins the $p$-player game, and let $R_{n,p,k}$ denote the probability that this happens in exactly $n$ throws (the letter $R$ stands for {\em refined}). 
With the assumption from Remark~\ref{rem:uniform}, we further refine these values for $x\in[0,1]$ as follows. 
Let $P_{p,k}(x)$ be the probability that, at a point in the game when $p$ players remain and the best throw so far is $x$ (with the convention that $x=1$ at the beginning), the $k$th player (in the order in which players are about to throw) wins the game. Let $R_{n,p,k}(x)$ be the probability that this happens with exactly $n$ additional throws. Note that $P_{p,k}(1)=P_{p,k}$ and $R_{n,p,k}(1)=R_{n,p,k}$.

\subsection{Computing the probability of winning}

In this section we will give a recursive formula to compute generating functions for the probabilities $P_{p,k}(x)$ and $R_{n,p,k}(x)$. Let us start by describing a recurrence for  the values $R_{n,p,k}(x)$, which refines Lemma~\ref{lem:Qnp}.

\begin{lemma}\label{lem:Rnpk}
For $n\ge1$, $p\ge2$ and $1\le k\le p$, the polynomials $R_{n,p,k}(x)$ satisfy the recurrence
\begin{align*}
R_{n,p,1}(x)&=\int_0^x R_{n-1,p,p}(v)\,dv,\\ 
R_{n,p,k}(x)&=(1-x)R_{n-1,p-1,k-1}(x)+\int_0^x R_{n-1,p,k-1}(v)\,dv \qquad \text{for }2\le k\le p,
\end{align*}
with initial conditions $R_{0,1,1}(x)=1$, $R_{n,1,1}(x)=0$ for $n\ge1$, and $R_{0,p,k}(x)=0$ for $2\le k\le p$.
\end{lemma}

\begin{proof}
The initial conditions are clear since the game requires no further throws if and only if $p=1$.

Let us prove the first equation. For the first player to win the $p$-player game where the distance to beat is $x$, first his dart must land at distance $v$ for some $0\le v< x$. After moving to the end of the line, he must then win the subsequent $p$-player game, where the distance to beat is $v$, and where he now plays the role of player $p$, with one fewer throw left.

The second equation is proved similarly, but now there are two options for the first throw. With probability $1-x$, it is a losing throw and the game reduces to a $(p-1)$-player game, where player $k$ (for $2\le k\le p)$ in the original game becomes player $k-1$ in the new game, and the distance to beat is still $x$. 
The other option is that the first throw lands at distance $v$ for some $0\le v<x$, in which case the game reduces to a $p$-player game, where again player $k$ (for $2\le k\le p)$ in the original game becomes player $k-1$ in the new game, and the distance to beat is now $v$.

Finally, the recurrence and the initial conditions imply that $R_{n,p,k}(x)$ is always a polynomial of degree at most~$n$.
\end{proof}

The polynomials $R_{n,p,k}(x)$ for small values of $n$ and $p$ are given in Table~\ref{tab:Rnpk}. 
By definition, \[\sum_{k=1}^p R_{n,p,k}(x)=Q_{n,p}(x),\] whose generating function 
is given in Theorem~\ref{thm:Qxuz}. In particular, adding the $k$ equalities in Lemma~\ref{lem:Rnpk}, we recover Lemma~\ref{lem:Qnp}.

\begin{table}[htb]
\[\begin{array}{c|cc|ccc}
&\multicolumn{2}{c|}{p=2}&\multicolumn{3}{c}{p=3}\\ \hline
n&k=1&k=2&k=1&k=2&k=3\\ \hline
1& 0&1-x&0&0&0\\
2& x-\dfrac{x^2}{2}&0&0&0&(1-x)^2\\
3& 0&\dfrac{x^2}{2}-\dfrac{x^3}{6}&x-x^2+\dfrac{x^3}{3}&x-\dfrac{3x^2}{2}+\dfrac{x^3}{2}&0\\
4& \dfrac{x^3}{6}-\dfrac{x^4}{24}&0&0&\dfrac{x^2}{2}-\dfrac{x^3}{3}+\dfrac{x^4}{12}&x^2-\dfrac{7x^3}{3}+\dfrac{7x^4}{24}
\end{array}\]
\[\begin{array}{c|cccc}
&\multicolumn{4}{c}{p=4}\\ \hline
n&k=1&k=2&k=3&k=4\\ \hline
1& 0&0&0&0\\
2& 0&0&0&0\\
3& 0&0&0&(1-x)^3\\
4& \dfrac{x(2-x)(2-2x+x^2)}{4}&\dfrac{x(1-x)(3-3x+x^2)}{3}&\dfrac{x(1-x)^2(2-x)}{2}&0
\end{array}\]
\caption{The polynomials $R_{n,p,k}(x)$ giving the probability that, when $p$ players remain and the distance to beat is $x\in[0,1]$, the $k$th player wins the game after $n$ further throws, for $n\le4$ and $p\le4$.}
\label{tab:Rnpk}
\end{table}

We can easily translate Lemma~\ref{lem:Rnpk} into a system of equations satisfied by the generating functions
$$R_{p,k}(x,z):=\sum_{n\ge0} R_{n,p,k}(x) z^n,$$
by simply multiplying each of the equations in Lemma~\ref{lem:Rnpk} by $z^n$ and summing over $n\ge0$.

\begin{lemma}\label{lem:Rpk}
For $p\ge2$, we have
\begin{align*}
R_{p,1}(x,z)&=z\int_0^x R_{p,p}(v,z)\,dv,\\ 
R_{p,k}(x,z)&=(1-x)zR_{p-1,k-1}(x,z)+z\int_0^x R_{p,k-1}(v,z)\,dv \qquad \text{for }2\le k\le p,
\end{align*}
where $R_{1,1}(x,z)=1$.
\end{lemma}

Setting $z=1$ in Lemma~\ref{lem:Rpk}, which corresponds to disregarding the number of throws and only keeping track of the winning player, yields the following equations for $P_{p,k}(x)=R_{p,k}(x,1)=\sum_{n\ge0}R_{n,p,k}(x)$, first discovered by Niedermaier~\cite{Nie}.

\begin{corollary}\label{cor:Ppk}
For $p\ge2$, we have
\begin{align*}
P_{p,1}(x)&=\int_0^x P_{p,p}(v)\,dv,\\
P_{p,k}(x)&=(1-x)P_{p-1,k-1}(x)+\int_0^x P_{p,k-1}(v)\,dv \qquad \text{for }2\le k\le p,
\end{align*}
where $P_{1,1}(x,z)=1$.
\end{corollary}

For small values of $p$, it is not hard to solve the systems of differential equations given by Lemma~\ref{lem:Rpk} and Corollary~\ref{cor:Ppk} to obtain expressions for $R_{p,k}(x,z)$ and $P_{p,k}(x)$, respectively.
The main result in this section is a general recursive procedure to compute these quantities for arbitrary $p\ge1$. Define the generating functions
$$R_p(x,t,z):=\sum_{k=1}^p R_{p,k}(x,z) t^k,\quad P_p(x,t):=\sum_{k=1}^p P_{p,k}(x) t^k.$$ 
Clearly, $R_p(x,t,1)=P_p(x,t)$, and $P_p(x,1)=\sum_{k=1}^p P_{p,k}(x)=1$.
The next theorem gives a method to compute $R_p(x,t,z)$.

\begin{theorem}\label{thm:Rp} 
For $p\ge2$, the generating functions $R_p(x,t,z)$ satisfy the recurrence
\begin{equation}\label{eq:Rp}
R_p(x,t,z)=\frac{t^{p+1}-t}{p}\sum_{j=0}^{p-1}\frac{\tR_p(x,\zeta_p^j,z)}{t-\zeta_p^j},
\end{equation}
where $\zeta_p=e^{2\pi i/p}$ is a primitive $p$th root of unity, and 
\begin{equation}\label{eq:tRp}\tR_p(x,t,z)=\left(z^{p-1}+ tz\int_0^x \frac{\partial }{\partial v}\left[(1-v)R_{p-1}(v,t,z)\right]e^{-vtz}\,dv\right)e^{xtz},\end{equation}
with initial condition $R_1(x,t,z)=t$. 
\end{theorem}

\begin{proof}
The initial condition $R_1(x,t,z)=R_{1,1}(x,z)\,t=t$ is clear from Lemma~\ref{lem:Rpk}.
Differentiating the equations in Lemma~\ref{lem:Rpk} with respect to $x$, we obtain 
\begin{align*}
\frac{\partial }{\partial x}R_{p,1}(x,z)&=z R_{p,p}(x,z),\\ 
\frac{\partial}{\partial x}R_{p,k}(x,z)&=z\,\frac{\partial }{\partial x}\left[(1-x)R_{p-1,k-1}(x,z)\right]+z R_{p,k-1}(x,z) \qquad \text{for }2\le k\le p.
\end{align*}
Multiplying the $k$th equation by $t^k$, for $1\le k\le p$, and adding them up, we get 
\[
\frac{\partial }{\partial x}\!\left[\sum_{k=1}^p R_{p,k}(x,z)\,t^k\right]=z\,\frac{\partial }{\partial x}\!\left[(1-x)\sum_{k=2}^p R_{p-1,k-1}(x,z)\,t^k\right]+zR_{p,p}(x,z)\,t+z\sum_{k=2}^p R_{p,k-1}(x,z)\,t^k,\]
which can be written more succinctly as
\begin{equation}\label{eq:RpRpp}\frac{\partial }{\partial x} R_p(x,t,z)=tz\,\frac{\partial }{\partial x}\left[(1-x)R_{p-1}(x,t,z)\right]+tzR_p(x,t,z)+(t-t^{p+1})zR_{p,p}(x,z).\end{equation}
One has the initial condition $R_p(0,t,z)=t^pz^{p-1}$, coming from the fact that, when $x=0$ and there are $p$ players left, the $p$th player wins after $p-1$ throws, where each of the first $p-1$ players is eliminated.

Equation~\eqref{eq:RpRpp} relates $R_p$ and $R_{p-1}$, but unfortunately it also has a term involving $R_{p,p}$. To get around this issue, first note that, by definition, $R_p(x,t,z)$ is a polynomial in $t$ of degree $p$, whose coefficients $R_{p,k}(x,z)$ are formal power series in $\Q[[x,z]]$.
In the quotient of the polynomial ring $\Q[[x,z]][t]$ by the ideal $\langle t^p-1\rangle$,
the polynomial $R_p(x,t,z)$ is congruent to $\tR_p(x,t,z)$, defined via the differential equation
\[ \frac{\partial }{\partial x} \tR_p(x,t,z)=tz\,\frac{\partial }{\partial x}\left[(1-x)R_{p-1}(x,t,z)\right]+tz\tR_p(x,t,z), \]
with initial condition $\tR_p(0,t,z)=z^{p-1}$. This differential equation has been obtained from~\eqref{eq:RpRpp} by canceling terms in the ideal $\langle t^p-1\rangle$. Solving this equation yields
\[
\tR_p(x,t,z)=\left(z^{p-1}+ tz\int_0^x \frac{\partial }{\partial v}\left[(1-v)R_{p-1}(v,t,z)\right]e^{-vtz}\,dv\right)e^{xtz},
\]
which agrees with equation~\eqref{eq:tRp}.

Thus, for each $1\le k\le p$, the coefficient of $t^k$ in $R_p(x,t,z)$ is obtained by combining all the coefficients of terms $t^n$ where $n\equiv k\bmod p$ in $\tR_p(x,t,z)$, that is,
\begin{equation}\label{eq:extract} R_{p,k}(x,z)=[t^k]R_p(x,t,z)=\sum_{n\equiv k\bmod p} [t^n] \tR_p(x,t,z)=\frac{1}{p}\sum_{j=0}^{p-1}\zeta_p^{-kj}R_p(x,\zeta_p^j,z),\end{equation}
where we used the fact that, if $A(t)=\sum_{n\ge0} a_{n} t^n$, then $$\frac{1}{p}\sum_{j=0}^{p-1} \zeta_p^{-kj} A(\zeta_p^j)=\frac{1}{p}\sum_{n\ge0} a_n \sum_{j=0}^{p-1} \zeta_p^{(n-k)j}=\sum_{n\equiv k\bmod p} a_n.$$

Finally, by equation~\eqref{eq:extract},
 $$R_p(x,t,z)=\sum_{k=1}^p R_{p,k}(x,z) t^k=\frac{1}{p}\sum_{j=0}^{p-1} R_p(x,\zeta_p^j,z) \sum_{k=1}^p \zeta_p^{-kj}t^k=\frac{1}{p}\sum_{j=0}^{p-1} R_p(x,\zeta_p^j,z) 
 \frac{t^{p+1}-t}{t-\zeta_p^j},$$
 proving equation~\eqref{eq:Rp}
\end{proof}

\subsection{Examples}\label{sec:examples}

In this section we show some expressions for $R_{p,k}$ and $P_{p,k}$ obtained from Theorem~\ref{thm:Rp} for small values of~$p$. Some of these had originally been obtained by Niedermaier~\cite{Nie}.

\begin{corollary}
For the $2$-player game,
\[ R_{2,1}(x,z)=1+\frac{z-1}{2}\e^{xz}-\frac{z+1}{2}\e^{-xz},\qquad R_{2,2}(x,z)=\frac{z-1}{2}\e^{xz}+\frac{z+1}{2}\e^{-xz}.\]
Setting $z=1$,
\[ P_{2,1}(x)=1-\e^{-x}, \qquad P_{2,2}(x)=\e^{-x},\]
and setting $x=1$,
\[ P_{2,1}=1-\e^{-1}\approx0.6321205588, \qquad P_{2,2}=\e^{-1}\approx0.3678794412.\]
\end{corollary}

\begin{corollary}
For the $3$-player game,
\delete{
\begin{align*}R_{3,1}(x,z)&=
\frac{ 2 {\mathrm e}^{\frac{x z}{2}}
\left( (z -1) \left(z +\frac{5}{3}\right) \cos\left(\frac{x z \sqrt{3}}{2}\right)
+\frac{ (z^{2}-4 z -3) \sin \left(\frac{x z \sqrt{3}}{2}\right)}{\sqrt{3}}\right)
-(z -1) \left(-\frac{8}{3}+\left(x^{2}-2 x \right) z^{2}+\left(x -1\right) z \right) {\mathrm e}^{2 x z}
-(-3 x +3) (z^{2}+z)}
{6} {\mathrm e}^{-x z} +1
\end{align*}
}
\begin{align*}
P_{3,1}(x)&=1+(x-1)\e^{-x}-\frac{2}{\sqrt{3}}\e^{-\frac{x}{2}}\sinsq,\\
P_{3,2}(x)&=-\e^{-x}+\e^{-\frac{x}{2}}\left(\frac{1}{\sqrt{3}}\sinsq+\cossq\right),\\
P_{3,3}(x)&=(2-x)\e^{-x}+\e^{-\frac{x}{2}}\left(\frac{1}{\sqrt{3}}\sinsq-\cossq\right).
\end{align*}
Setting $x=1$,
\[
P_{3,1}=1-\frac{2}{\sqrt{3\e}}\sin\!\left(\frac{\sqrt{3}}{2}\right)\approx0.4664928047,\quad
P_{3,2}\approx0.2918207124,\quad
P_{3,3}\approx0.2416864833.
\]
\end{corollary}

\begin{corollary}
For the $4$-player game,
\begin{align*}
P_{4,1}(x)&=1+\frac{5}{4}(\cos x-\sin x)-\frac{3x^2-8x+5}{4}\e^{-x}\\
&\quad+\e^{-\frac{x}{2}}\left(\sqrt{3}\left(x - \frac{1}{3}\right)\sinsq-(x + 1)\cossq\right),\\
P_{4,2}(x)&=\frac{5}{4}(\cos x+\sin x)-\frac{x^2-6x+3}{4}\e^{-x}-\e^{-\frac{x}{2}}
\left(\frac{x + 3}{\sqrt{3}}\sinsq+(x+1)\cossq\right),\\
P_{4,3}(x)&=\frac{5}{4}(\sin x-\cos x)+\frac{x^2-7}{4}e^{-x}+\e^{-\frac{x}{2}}
\left(3\cossq-\frac{2x+1}{\sqrt{3}}\sinsq\right),\\
P_{4,4}(x)&=-\frac{5}{4}(\sin x+\cos x)+\frac{9x^2-42x+39}{12}\e^{-x} +\e^{-\frac{x}{2}}
\left(\frac{5}{\sqrt{3}}\sinsq+(2x-1)\cossq\right).
\end{align*}
Setting $x=1$,
\[
P_{4,1}\approx0.3711532353,\quad P_{4,2}\approx0.242188553,\quad P_{4,3}\approx0.2032205606, \quad P_{4,4}\approx0.1834376522.
\]
\end{corollary}

For the $5$-player game, the plot of the functions $P_{5,k}(x)$ for $1\le k\le 5$ over $x\in[0,1]$ appears in Figure~\ref{fig:P12345}. The formulas resulting from Theorem~\ref{eq:Rp} are long, so we only include the expression for $P_{5,1}(1)$ as an example:
\begin{multline*}
P_{5,1}=
1-\frac{25}{2}\cos(1)
 -{\frac { \sqrt {10-2\,\sqrt {5}}\left( 3\,\sqrt {5}+2\right) }{10}\, \e^{\frac{\sqrt{5}-1}{4}}\sin\! \left( \frac{\sqrt {10-2\,\sqrt {5}} \left( \sqrt {5}+1 \right)}{8}  \right) }\\
 +{\frac {  3\,\sqrt {5}+7 }{2}\, 
\e^{\frac{\sqrt {5}-1}{4}}\cos\!\left(\frac{\sqrt {10-2\,\sqrt {5}}
 \left( \sqrt {5}+1 \right)}{8}  \right) }  
-{\frac {3\,\sqrt {5}-7 }{2}\, \e^{\frac{-\sqrt{5}-1}{4}}\cos\! \left( \frac{\sqrt {10-2\,\sqrt {5}}}{4} \right) } 
 \\ 
-{\frac {\sqrt {10-2\,\sqrt {5}}\left( \sqrt {5}+13 \right)}{20}\,\e^{\frac{-\sqrt{5}-1}{4}}
  \sin\! \left( \frac{\sqrt {10-2\,\sqrt {5}}}{4} \right) } 
 +5\,\e^{-\frac12}\left(\cos\! \left( \frac{\sqrt {3}}{2} \right) +\frac{1}{\sqrt{3}}\sin\!\left(\frac{\sqrt {3}}{2} \right) \right).
\end{multline*}
\delete{
\begin{align*}P_{5,1}=
1-\frac{25}{2}\cos(1) +{\e^{-1/2}} 
&\left( {\frac {  3\,\sqrt {5}+7 }{2}\, 
\e^{\frac{1+\sqrt {5}}{4}}\cos\!\left(\frac{\sqrt {10-2\,\sqrt {5}}
 \left( \sqrt {5}+1 \right)}{8}  \right) } \right. \\ 
& -{\frac { \sqrt {10-2\,\sqrt {5}} \left( 3\,\sqrt {5}+2\right) }{2}\, \e^{\frac{1+\sqrt{5}}{4}}\sin\! \left( \frac{\sqrt {10-2\,\sqrt {5}} \left( \sqrt {5}+1 \right)}{8}  \right) }\\
&-{\frac {3\,\sqrt {5}-7 }{10}\, \e^{\frac{1-\sqrt{5}}{4}}\cos\! \left( \frac{\sqrt {10-2\,\sqrt {5}}}{4} \right) }\\
&-{\frac {\sqrt {10-2\,\sqrt {5}}\left( \sqrt {5}+13 \right)}{20}\,\e^{\frac{1-\sqrt{5}}{4}}
  \sin\! \left( \frac{\sqrt {10-2\,\sqrt {5}}}{4} \right) }\\
 &\left. 
 +5\cos \left( \frac{\sqrt {3}}{2} \right) +\frac{5}{\sqrt{3}}\sin \left(\frac{\sqrt {3}}{2} \right) \right)
\end{align*}}
The numerical values of $P_{5,k}$ are approximately
\begin{gather}
P_{5,1}\approx0.3088745194, \quad P_{5,2}\approx0.2071760032, \quad P_{5,3}\approx0.1754708034,\\
P_{5,4}\approx 0.1592777163,\quad P_{5,5}\approx 0.1492009703.
\end{gather}

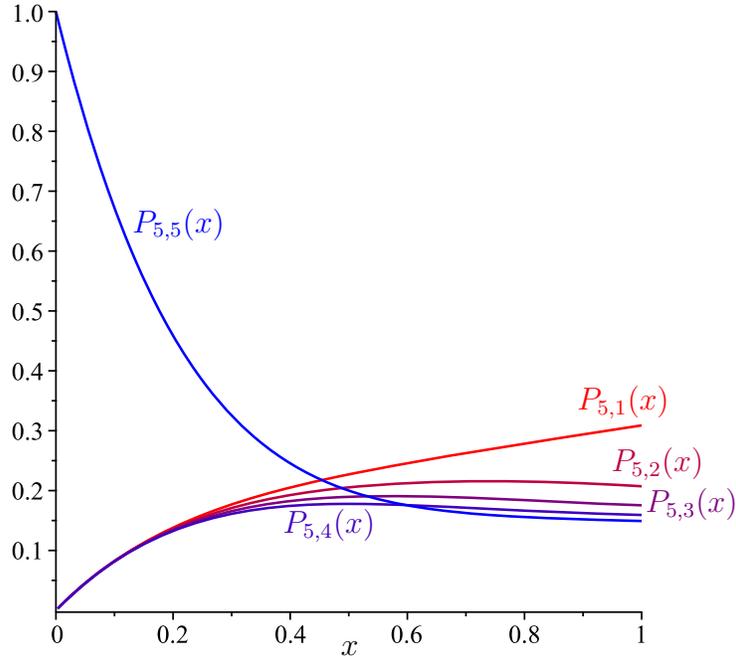
\begin{figure}[htb]
\centering
\begin{tikzpicture}
 \node[anchor=south west,inner sep=0] (image) at (0,0) {\includegraphics[width=.55\textwidth]{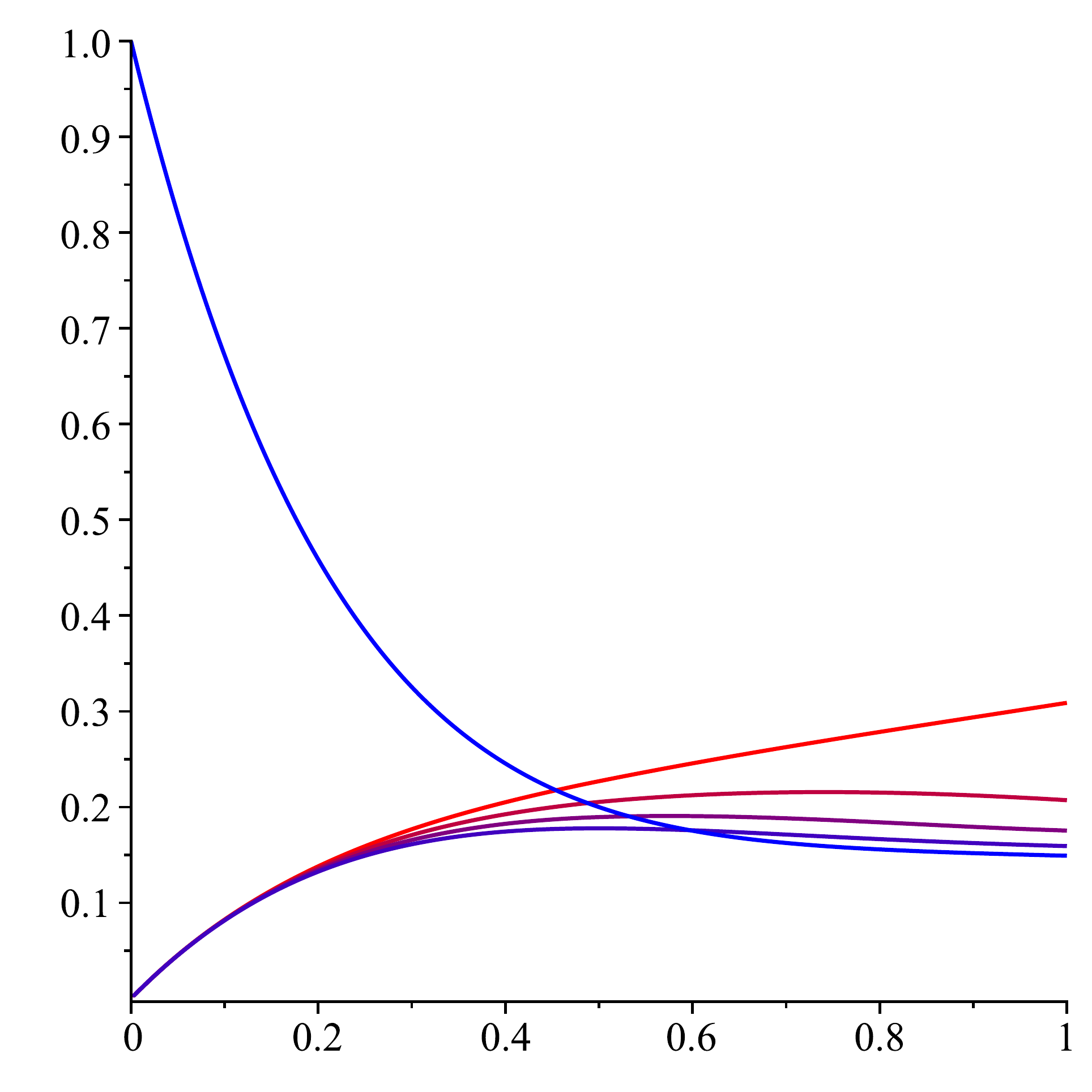}};
    \begin{scope}[x={(image.south east)},y={(image.north west)}]
    \draw (.3,.65) node {\color[rgb]{0,0,1} $P_{5,5}(x)$};
    \draw (.52,.21) node {\color[rgb]{0.25,0,.75} $P_{5,4}(x)$};
    \draw (1.05,.24) node {\color[rgb]{0.5,0,.5} $P_{5,3}(x)$};
    \draw (1,.30) node {\color[rgb]{0.75,0,.25} $P_{5,2}(x)$};
    \draw (.95,.39) node {\color[rgb]{1,0,0} $P_{5,1}(x)$};
    \draw (.55,0.03) node {$x$};
    \end{scope}
\end{tikzpicture}
\caption{The plots of the functions $P_{5,k}(x)$ for $1\le k\le 5$.}
\label{fig:P12345}
\end{figure}

\subsection{A recurrence for the number of permutations encoding games with a given winner} 

For $1\le k\le p$ and $n\ge0$,  let $\W_{n,p,k}$ denote the set of permutations in $\S_n$ encoding $p$-player games in which player $k$ is the winner, and let $w_{n,p,k}=|\W_{n,p,k}|$. 
By equation~\eqref{eq:complete}, we have $\sum_{k=1}^p w_{n,p,k}=(n-1)c(n-1,n-p+1)$ for $n\ge p\ge2$, since this is the number of permutations in $\S_n$ encoding complete $p$-player games. 
It is also clear that
\begin{equation}\label{eq:Rw} R_{n,p,k}=\frac{w_{n,p,k}}{n!}. \end{equation}
In particular, the numbers $w_{n,p,k}$ can be computed by first using the recurrence in Lemma~\ref{lem:Rnpk} to compute $R_{n,p,k}(x)$, and then substituting at $x=1$ and dividing by $n!$. 

Next we give an alternative recurrence to compute these numbers. Since the only permutation of length $n\le1$ that encodes a complete game is the empty permutation (corresponding to the trivial $1$-player game), we will assume in the rest of this section that $n\ge2$.
Instead of the variable $x$, let us consider an additional parameter that keeps track of the first entry of the permutation. For $1\le s\le n$, define
$$\W_{n,p,k,s}=\{\pi\in\W_{n,p,k}:\pi_1=s\}$$ 
and $w_{n,p,k,s}=|\W_{n,p,k,s}|$, so that $w_{n,p,k}=\sum_{s=1}^n w_{n,p,k,s}$.

\begin{theorem}
For $n\ge3$, $p\ge2$, $1\le k\le p$ and $1\le s\le n$, the numbers $w_{n,p,k,s}$ 
satisfy the recurrence
\begin{equation}\label{eq:w}
w_{n,p,k,s}=\sum_{s'=1}^{s-1} w_{n-1,p,\overline{k-1},s'}+(n-s)w_{n-1,p-1,\hat{k},s},
\end{equation}
where $\overline{k-1}\in[p]$ is congruent to $k-1$ modulo $p$,  and
$$\hat{k}=\begin{cases} k-1 & \text{for $k\ge3$},\\ 1 & \text{for $k=1$.} \end{cases}
$$
For $k=2$, the value $\hat{k}$ is not defined and the right summand in~\eqref{eq:w} does not appear.
The initial conditions are $w_{n,p,k,s}=0$ whenever $p=1$ or $n=2$, except that $w_{2,2,1,1}=1$.
\end{theorem}

\begin{proof}
The initial condition when $p=1$ is due to the fact that the $1$-player game has length $n=0$ and we are assuming $n\ge2$.
The values for $n=2$ are easy to check, since the only permutation of length $2$ encoding a complete game is $12$.
Now suppose that $n\ge3$, and let $\pi\in\W_{n,p,k,s}$ and $s'=\pi_2$.

If $s>s'$, then the permutation $\pi'$ obtained from $\pi$ by removing $\pi_1$ (and subtracting $1$ from the values larger than $\pi_1$) satisfies $\pi'\in \W_{n-1,p,\overline{k-1},s'}$. Indeed, by removing the first throw $\pi_1$, the resulting game proceeds like the original game but with the role of player $k$ now being played by player $\overline{k-1}$. Conversely, given 
$\pi'\in \W_{n-1,p,\overline{k-1},s'}$, one can insert any value $s>s'$ at the beginning and shift the larger values up by one accordingly to obtain a permutation in $\W_{n,p,k,s}$. This construction gives a bijection between $\{\pi\in\W_{n,p,k,s}:\pi_2=s'\}$ and $\W_{n-1,p,\overline{k-1},s'}$.

If $s<s'$, then the permutation $\pi'$ obtained from $\pi$ by removing $\pi_2$ (and subtracting $1$ from the values larger than $\pi_2$) satisfies $\pi'\in \W_{n-1,p-1,\hat{k},s}$. Indeed, by removing the throw $\pi_2$ where player 2 was eliminated, the resulting game proceeds like the original game but with the role of player $j$ (for each $j\ge3$) now being played by player $j-1$. 
Conversely, given 
$\pi'\in \W_{n-1,p-1,\hat{k},s}$, one can insert any value $s'>s$ right after the first entry and shift the larger values up by one accordingly to obtain a permutation in $\W_{n,p,k,s}$. This construction gives a bijection between $\{\pi\in\W_{n,p,k,s}:\pi_2=s'\}$ and $\W_{n-1,p-1,\hat{k},s}$. There are $n-s$ possible values for $s'$ in this case, so the recurrence follows.
\end{proof}

\section{Monotonicity of winning probabilities}\label{sec:monotone}

The numerical values for the probabilities in Section~\ref{sec:examples}, as well as basic intuition, suggest that each player has higher changes of winning than the next.
In this section we prove that this is always the case.

\begin{proposition}\label{prop:monotone}
For $p\ge2$ and $k\in[p-1]$,
$$P_{p,k}>P_{p,k+1}.$$
\end{proposition}

\begin{proof}[Probabilistic proof]
For $1\le j\le p$ and $i\ge1$, consider the random variable $X_{j,i}$  that describes the distance to the center of the $i$th throw of player $j$ in the game, if such throw is needed. Each array of values $\{X_{j,i}\}_{j\in[p],i\ge1}$ yields some winner, but note that only the values up until the throw where each player is eliminated are used. 

We will describe a measure-preserving injection between arrays for which player $k+1$ is the winner to arrays for which player $k$ is the winner.
Suppose that an array $\{X_{j,i}\}$ yields player $k+1$ as the winner.  Let $\{X'_{j,i}\}$ be the array obtained by switching rows $k$ and $k+1$, that 
is, letting $X'_{j,i}=X_{s_k(j),i}$ for all $i$ and $j$, where $s_k(k)=k+1$, $s_k(k+1)=k$, and $s_k(j)=j$ for $j\notin\{k,k+1\}$.
Then the array $\{X'_{j,i}\}$ yields player $k$ as the winner. Indeed, since $X'_{k+1,1}=X_{k,1}>X_{k+1,1}=X'_{k,1}$, player $k+1$ is eliminated after his first throw, and the rest of the game proceeds in as the original game with player $k$ playing the role of player $k+1$, and thus winning.

It follows that $P_{p,k}\ge P_{p,k+1}$. To see that the inequality is strict, note that any array $\{X'_{j,i}\}$ for which player $k$ is the winner but player $k+1$ does not get eliminated right after his first throw does not arise as the image of the above transformation. Such an outcome occurs with positive probability, and so the inequality is strict.
\end{proof}

\begin{proof}[Combinatorial proof]
By equation~\eqref{eq:Rw} and the fact that $P_{p,k}=\sum_{n\ge0} R_{n,p,k}$, it is enough to show that
$$\sum_{n\ge0} \frac{w_{n,p,k}}{n!}>\sum_{n\ge0} \frac{w_{n,p,k+1}}{n!}.$$
Note that it is not true that $w_{n,p,k}\ge w_{n,p,k+1}$ in general, as can be seen already for $p=2$.
Instead, we will decompose each of the sets $\W_{n,p,k}$ and $\W_{n,p,k+1}$ as a union of two subsets, and use them to describe injections to prove the inequality.
We will assume that $n\ge p$, since otherwise $w_{n,p,k}=w_{n,p,k+1}=0$.

Let $\U_{n,p,k+1}\subseteq\W_{n,p,k+1}$ be the subset of permutations encoding games where the last throw does not belong to player $k$, and let
$\Uc_{n,p,k+1}=\W_{n,p,k+1}\setminus\U_{n,p,k+1}$.
Let $\V_{n,p,k}\subseteq\W_{n,p,k}$ be the subset of permutations encoding games where player $k+1$'s losing throw would still have been a losing throw even if it had taken place before the previous throw (which belongs to player $k$), and let $\Vc_{n,p,k}=\W_{n,p,k}\setminus\V_{n,p,k}$.

Next we construct injections 
\begin{equation}\label{eq:inj}
\U_{n,p,k+1}\to\V_{n,p,k} \qquad \text{and} \qquad \Uc_{n,p,k+1}\to\Vc_{n-1,p,k}\times\{2,3,\dots,n\}.
\end{equation}

Given $\pi\in\U_{n,p,k+1}$, suppose that $\pi_i$ encodes player $k$'s losing throw. Then $\pi_{i+1}$ encodes a throw by player $k+1$ (the winner), and thus $i+1$ is a left-to-right minimum of $\pi$. Let $\pi'$ be obtained from $\pi$ by swapping $\pi_i$ and $\pi_{i+1}$. Then $i$
is a left-to-right minimum of $\pi'$, since $\pi'_i=\pi_{i+1}$, corresponding to a throw of player $k$, and $\pi'_{i+1}$ is now player $k+1$'s losing throw. In the rest of the game encoded by $\pi'$, player $k$ plays the role that player $k+1$ played in $\pi$, and so player $k$ is now the winner. Additionally, player $k+1$'s losing throw $\pi'_{i+1}=\pi_i$ would still have been a losing throw if it had taken place before $\pi'_i$, so $\pi'\in\V_{n,p,k}$. The map $\pi\mapsto\pi'$ is clearly an injection from $\U_{n,p,k+1}$ to $\V_{n,p,k}$.

Given $\pi\in\Uc_{n,p,k+1}$, we know by definition that $\pi_n$ encodes player $k$'s losing throw. Suppose that $\pi_i$ encodes player $k$'s last good throw, and so $\pi_{i+1}$ encodes a good throw by player $k+1$, meaning that both $i$ and $i+1$ are left-to-right minima of $\pi$. Let $\pi''$ be the permutation obtained from $\pi$ by swapping $\pi_i$ and $\pi_{i+1}$, removing $\pi_n$, and subtracting one from the values larger than $\pi_n$ to get a permutation of $[n-1]$. Then $i$  is a left-to-right minimum of $\pi''$ which corresponds to a throw of player $k$, and $\pi''_{i+1}$ corresponds to player $k+1$'s losing throw, which implies that player $k$ is now the winner. In this case, player $k+1$'s losing throw $\pi''_{i+1}$ would not have been a losing throw if it had taken place before $\pi''_i$, so $\pi''\in\Vc_{n,p,k}$. The map $\pi\to(\pi'',\pi_n)$ is clearly an injection from $\Uc_{n,p,k+1}$ to $\Vc_{n-1,p,k}\times\{2,3,\dots,n\}$.

Letting $u_{n,p,k+1}=|\U_{n,p,k+1}|$, and defining $\uc$, $v$ and $\vc$ analogously, the injections~\eqref{eq:inj} imply that $u_{n,p,k+1}\le v_{n,p,k}$ and
$\uc_{n,p,k+1}\le (n-1) \vc_{n-1,p,k}$. Adding these inequalities,
$$w_{n,p,k+1}\le v_{n,p,k}+(n-1) \vc_{n-1,p,k},$$
from where
\begin{align*}
\sum_{n\ge p} \frac{w_{n,p,k+1}}{n!} & \le  \sum_{n\ge p} \frac{v_{n,p,k}}{n!}+ \sum_{n\ge p}\frac{(n-1) \vc_{n-1,p,k}}{n!}\\
& < \sum_{n\ge p} \frac{v_{n,p,k}}{n!}+ \sum_{n\ge p}\frac{\vc_{n-1,p,k}}{(n-1)!}=\sum_{n\ge p} \frac{v_{n,p,k}}{n!}+ \sum_{n\ge p}\frac{\vc_{n,p,k}}{n!}=\sum_n \frac{w_{n,p,k}}{n!}.
\qedhere
\end{align*}
\end{proof}

\section{Open questions}

We finish by listing a few possible avenues for further research.

\begin{itemize}
\item As the number of players tends to infinity, does the (properly rescaled) probability distribution $\{P_{p,k}\}_{k=1}^p$ converge? If so, what is the limit distribution?
\item Is there a closed form for the generating function $R(x,t,u,z)=\sum_{p\ge1} R_p(x,t,z)u^p$? Note that, by definition, $R(x,1,u,z)=Q(x,u,z)$, given by Theorem~\ref{thm:Qxuz}.
\item Consider other refinements, for example by keeping track of the last player eliminated, or, more generally, of the permutation describing the ranking of the players (the opposite of the order in which they are eliminated).
\item Consider variations of the game where players get better over time. For example, the permutation determined by the distances of the first $n$ throws could follow a Mallows distribution with parameter $q>1$.
\end{itemize}

\section*{Acknowledgements} 
We thank Andy Niedermaier for introducing this game, for obtaining some preliminary results, and for allowing us to include them here.

\end{document}